\documentclass[12pt]{article}

\usepackage[english]{babel}
\usepackage[a4paper]{geometry}
\usepackage{amssymb,amsmath,amsthm}
\usepackage{graphicx,cite,color}
\usepackage{longtable,pdflscape,booktabs,caption,multicol}

\usepackage[hyphens]{url}
\usepackage[colorlinks=true,citecolor=black,linkcolor=black,urlcolor=blue]{hyperref}
\usepackage[hyphenbreaks]{breakurl}

\usepackage{enumitem}
\usepackage{mathtools}

\theoremstyle{plain}
\newtheorem{theorem}{Theorem}
\newtheorem{lemma}[theorem]{Lemma}
\newtheorem{corollary}[theorem]{Corollary}
\newtheorem{proposition}[theorem]{Proposition}

\newtheorem{conjecture}[theorem]{Conjecture}
\newtheorem{problem}[theorem]{Problem}

\theoremstyle{definition}

\newtheorem{example}{Example}

\theoremstyle{remark}

\newcommand{\arxiv}[2]{\href{https://arxiv.org/abs/#1}{\texttt{arXiv:#1}} \texttt{[#2]}}
\newcommand{\doi}[1]{\url{https://doi.org/#1}}

\DeclareMathOperator{\beg}{beg}
\DeclareMathOperator{\inv}{inv}

\usepackage{float}
\usepackage{tikz}
\usepackage{subfig}
\usetikzlibrary{arrows.meta}
\usetikzlibrary{automata}

\tikzset{
    edge/.style={-{Latex[scale=1.2]}},
}

\usepackage{blkarray, bigstrut}
\usepackage{bm}

\def\dj{d\kern-0.4em\char"16\kern-0.1em}
\def\Dj{\hbox{\raise0.3ex\hbox{-}\kern-0.4em  D}}

\usepackage{authblk}

\title{On cubic polycirculant nut graphs}

\author[1,2,3]{Nino Bašić}
\author[1,4,5]{Ivan Damnjanović}

\affil[1]{FAMNIT, University of Primorska, Koper, Slovenia}
\affil[2]{IAM, University of Primorska, Koper, Slovenia}
\affil[3]{Institute of Mathematics, Physics and Mechanics, Ljubljana, Slovenia}
\affil[4]{Faculty of Electronic Engineering, University of Niš, Niš, Serbia}
\affil[5]{Diffine LLC, San Diego, California, USA}

\date{}

\begin{document}

\maketitle

\begin{abstract}
A nut graph is a nontrivial simple graph whose adjacency matrix contains a one-dimensional null space spanned by a vector without zero entries. Moreover, an $\ell$-circulant graph is a graph that admits a cyclic group of automorphisms having $\ell$ vertex orbits of equal size. It is not difficult to observe that there exists no cubic $1$-circulant nut graph or cubic $2$-circulant nut graph, while the full classification of all the cubic $3$-circulant nut graphs was recently obtained [\emph{Electron.\ J.\ Comb.\/}\ \textbf{31(2)} (2024), \#2.31]. Here, we investigate the existence of cubic $\ell$-circulant nut graphs for $\ell \ge 4$ and show that there is no cubic $4$-circulant nut graph or cubic $5$-circulant nut graph by using a computer-assisted proof. Furthermore, we rely on a  construction based approach in order to demonstrate that there exist infinitely many cubic $\ell$-circulant nut graphs for any fixed $\ell \in \{6, 7 \}$ or $\ell \ge 9$.
\end{abstract}



\section{Introduction}

A \emph{nut graph} is a nontrivial simple graph such that its adjacency matrix has a one-dimensional null space all of whose nonzero vectors contain no zero entry. These graphs were introduced by Sciriha and Gutman \cite{Sciriha1997, Sciriha1998_A, Sciriha1998_B, ScGu1998, Sciriha1999} and subsequently investigated in a series of papers \cite{Sciriha2007, Sciriha2008, FoGaGoPiSc2020, GaPiSc2023}. The chemical justification for studying nut graphs can be found in \cite{ScFo2007, ScFo2008, FoPiToBoSc2014, CoFoGo2018, FoPiBa2021} and for more results concerning them, see \cite{ScFa2021}. Furthermore, for any $\ell \in \mathbb{N}$, an \emph{$\ell$-circulant graph} is a graph that admits a cyclic group of automorphisms with $\ell$ vertex orbits of equal size. If $\ell = 1$, we will then omit this parameter and just refer to the $1$-circulant graphs as the circulant graphs.

The problem of finding all the orders that a $d$-regular nut graph can have for a fixed $d \in \mathbb{N}$ was first considered in \cite{GaPiSc2023}, where it was solved for $d \le 4$. This result was further improved in \cite{FoGaGoPiSc2020}, where the aforementioned existence problem for resolved up to $d \le 11$. Afterwards, a circulant graph based construction was applied \cite{BaKnSk2022} in order to show that there exists a $12$-regular nut graph of order $n \in \mathbb{N}$ if and only if $n \ge 16$. In a subsequent series of papers \cite{DaSt2022, Damnjanovic2023_FIL, Damnjanovic2024_AMC}, the circulant nut graph order--degree existence problem was fully resolved via the following result.

\begin{theorem}[\hspace{1sp}{\cite[Theorem 5]{Damnjanovic2024_AMC}}]\label{circulant_nut_th}
For each $n \in \mathbb{N}, d \in \mathbb{N}_0$, there exists a $d$-regular circulant nut graph of order $n$ if and only if $d > 0$, $4 \mid d$, $2 \mid n$, alongside $n \ge d + 4$ if $d \equiv_8 4$, and $n \ge d + 6$ if $8 \mid d$, as well as $(n, d) \neq (16, 8)$.
\end{theorem}

For more results on circulant and vertex-transitive nut graphs, see, e.g., \cite{BaFo2024, BaFoPi2024, Damnjanovic2023_ARX, Damnjanovic2024_DMC}. As a direct consequence of Theorem \ref{circulant_nut_th}, it follows that there is no cubic circulant nut graph and it is also straightforward to see that there exists no cubic $2$-circulant nut graph (see, e.g., \cite[Proposition 9]{DaBaPiZi2024}). In a recent paper \cite[Theorem 2]{DaBaPiZi2024}, the full classification of cubic $3$-circulant nut graphs was disclosed, thus showing that there are infinitely many such graphs. Here, we investigate the existence of cubic $\ell$-circulant nut graphs for $\ell \ge 4$, with the main result given in the next theorem.

\begin{theorem}\label{main_th}
    For any $\ell \in \{1, 2, 4, 5 \}$, there exists no cubic $\ell$-circulant nut graph, while for each $\ell \in \{3, 6, 7 \}$ or $\ell \ge 9$, there exist infinitely many cubic $\ell$-circulant nut graphs.
\end{theorem}

Throughout the rest of the paper, our focus will be to prove Theorem~\ref{main_th}. In Section \ref{sc_prel} we will introduce the notation to be used in the subsequent sections and we will also preview certain basic theoretical facts from graph theory and linear algebra. Section \ref{sc_maga} will serve to elaborate a technique that forms the basis of a computed-assisted search which can be applied to disprove the existence of various cubic polycirculant nut graphs. The main result of this section will be the proof of the nonexistence of cubic $\ell$-circulant nut graphs for $\ell = 4$ or $\ell = 5$.

Afterwards, Sections \ref{sc_7_circ} and \ref{sc_11_circ} will be used to demonstrate the existence of infinitely many cubic $7$-circulant nut graphs and cubic $11$-circulant nut graphs, respectively. Section \ref{sc_subdiv} will subsequently show how, under certain conditions, a given cubic $\ell$-circulant nut graph can be extended to a cubic $(\ell+3)$-circulant graph that is also a nut graph, thus completing the proof of Theorem \ref{main_th}. In Section \ref{sc_conclusion} we will disclose some open existence problems regarding the cubic polycirculant nut graphs, while the \texttt{SageMath} \cite{SageMath} code used to perform the computer-assisted search from Section \ref{sc_maga} can be found in \cite{GitHub}.

\section{Preliminaries}\label{sc_prel}

Although nut graphs are necessarily simple graphs and, as such, contain no loops or parallel edges, in order to be able to concisely represent polycirculant graphs, it is convenient to rely on pregraphs, which are allowed to have loops, parallel edges or even semi-edges. As done so in \cite{MaMaPo2004}, we will take a \emph{pregraph} to be an ordered quadruple $X = (D, V; \beg, \inv)$ where $D$ and $V \neq \varnothing$ are the disjoint finite sets of darts and vertices, respectively, $\beg \colon D \to V$ is the mapping which assigns the initial vertex $\beg a$ to each dart $a$, while $\inv \colon D \to D$ is the involution that interchanges each dart $a$ with its inverse dart $\inv a$. Thus, the \emph{adjacency matrix} $A(X)$ of the pregraph $X$ can naturally be defined as the symmetric $\mathbb{R}^{V(X) \times V(X)}$ matrix for which $A(X)_{xy}$ equals the number of darts $a \in D(X)$ such that $\beg a = x$ and $\beg (\inv a) = y$. Also, for each real matrix $B$, we will use $|B|$ to denote the matrix containing the absolute values of its corresponding entries.

Now, for a given group $(\Gamma, *)$, we will take a \emph{$\Gamma$-voltage pregraph} to be an ordered pair $(X, \varphi)$ where $X = (D, V; \beg, \inv)$ is a pregraph, while $\varphi \colon D \to \Gamma$ is a mapping such that $\varphi(\inv a) = \left(\varphi(a)\right)^{-1}$. Moreover, we will say that the aforementioned voltage pregraph gives rise to the pregraph $X_1 = (D_1, V_1; \beg_1, \inv_1)$ such that
\begin{enumerate}[label=\textbf{(\roman*)}]
    \item $V_1 = V \times \Gamma$ and $D_1 = D \times \Gamma$;
    \item $\beg_1(a, \gamma) = (\beg a, \gamma)$ for each $a \in D, \gamma \in \Gamma$;
    \item $\inv_1(a, \gamma) = (\inv a, \varphi(a) * \gamma)$ for each $a \in D, \gamma \in \Gamma$.
\end{enumerate}
We will also refer to $X_1$ as the \emph{derived graph} of the $\Gamma$-voltage pregraph $(X, \varphi)$. For further insight into the theory regarding pregraphs and voltage graphs, please refer to \cite[Section 3.5]{PiSe2013} and \cite{MaMaPo2004, MaNeSko2000, Pisanski2007, PoTo2020}.

A \emph{simple graph} can now be viewed as a pregraph without semi-edges, loops or parallel edges. For a given simple graph (or pregraph) $X$, we will use $V(X)$, $A(X)$ and $\mathcal{N}(X)$ to denote its set of vertices, adjacency matrix and the null space of its adjacency matrix, respectively. Furthermore, the next lemma contains a key result that we shall consistently make use of.

\begin{lemma}[\hspace{1sp}{\cite[Lemma 4]{DaBaPiZi2024}}]\label{local_condition_lemma}
For any simple graph $G$ and an arbitrary vector $u \in \mathbb{R}^{V(G)}$, we have that $u \in \mathcal{N}(G)$ holds if and only if
\begin{equation}\label{local_cond}
    \sum_{y \sim x} u(y) = 0
\end{equation}
is satisfied for every vertex $x \in V(G)$.
\end{lemma}

Equation~\eqref{local_cond} is usually referred to as the \emph{local condition} corresponding to the vertex $x$. We will call a vector without zero entries a \emph{full vector}, and otherwise, we will call it a \emph{nonfull vector}. We now present two lemmas that deal with the well-known properties of nut graphs.

\begin{lemma}[\hspace{1sp}{\cite{ScGu1998}}]\label{basic_nut_lemma}
    Each nut graph is connected, nonbipartite and leafless.
\end{lemma}
\begin{lemma}[\hspace{1sp}{\cite[p.\ 135]{Cvetkovic1995}}]\label{orbit_lemma}
    Let $G$ be a nut graph and let $u \in \mathcal{N}(G)$. If the vertices $a_j, \, j \in \mathbb{Z}_t$, form an orbit of size $t$ of a given automorphism $\pi \in \mathrm{Aut}(G)$, so that $\pi(a_j) = a_{j+1}$, then either $u(a_j)$ is constant for $j \in \mathbb{Z}_t$, or $t$ is even and we have $u(a_j)$ = $(-1)^j \, u(a_0)$ for every $0 \le j \le t-1$.
\end{lemma}

As a direct consequence of Lemma \ref{orbit_lemma}, we conclude that for an arbitrary nut graph $G$ and any fixed nonzero vector from $\mathcal{N}(G)$, all the vector entries corresponding to the vertices from the same orbit must have the same absolute value. We will refer to this value as the \emph{magnitude} of the given orbit. It is straightforward to establish that the ratios of orbit magnitudes are fixed with respect to each selected nonzero $\mathcal{N}(G)$ vector.

Since we will deal with circulant matrices in the remainder of the paper, it is useful to bear in mind that the spectrum of any circulant matrix of order $n \in \mathbb{N}$
\[
    \begin{bmatrix}
        c_0 & c_1 & c_2 & \cdots & c_{n-1}\\
        c_{n-1} & c_0 & c_1 & \cdots & c_{n-2}\\
        c_{n-2} & c_{n-1} & c_0 & \cdots & c_{n-3}\\
        \vdots & \vdots & \vdots & \ddots & \vdots\\
        c_1 & c_2 & c_3 & \dots & c_0
    \end{bmatrix}
\]
can be obtained via the expression
\[
    c_0 + c_1 \zeta + c_2 \zeta^2 + \cdots + c_{n-1} \zeta^{n-1}
\]
by letting $\zeta \in \mathbb{C}$ range over all the $n$-th roots of unity. Moreover, for any such $\zeta$, the eigenvector $\begin{bmatrix} 1 & \zeta & \zeta^2 & \zeta^3 & \cdots & \zeta^{n-1} \end{bmatrix}^\intercal$ corresponds to the eigenvalue $\sum_{j=0}^{n-1} c_j \zeta^j$. For the proof of this well-known fact, please see \cite[Section~3.1]{Gray}.

We end the section by pointing out a few facts regarding the cyclotomic polynomials. For each $b \in \mathbb{N}$, we define the \emph{cyclotomic polynomial $\Phi_b(x)$} as
\[
    \Phi_b(x) = \prod_{\xi}(x - \xi) ,
\]
where $\xi \in \mathbb{C}$ ranges over the primitive $b$-th roots of unity. It is well known that for every $b \in \mathbb{N}$, the $\Phi_b(x)$ polynomial has integer coefficients and it is irreducible in $\mathbb{Q}[x]$ (see, e.g., \cite{Jameson} and the references therein). For this reason, any given $\mathbb{Q}[x]$ polynomial contains a primitive $b$-th root of unity among its roots if and only if it is divisible by $\Phi_b(x)$. We rely on this observation in Sections \ref{sc_7_circ} and~\ref{sc_11_circ}.

\section{Cubic 4- and 5-circulant nut graphs}\label{sc_maga}

The goal of the present section will be to outline an algorithm for inspecting the potential nonexistence of cubic $\ell$-circulant nut graphs for a fixed value of $\ell \ge 3$. We begin by observing that a cubic $\ell$-circulant graph of order $n$ is isomorphic to the derived graph of some cubic $\mathbb{Z}_{n / \ell}$-voltage pregraph on $\ell$ vertices. Thus, it is natural to inspect the existence of cubic $\ell$-circulant nut graphs by performing a search over the finitely many pregraphs of order $\ell$.

For a given pregraph $X$, we will say that its \emph{underlying graph} is the graph obtained from $X$ by removing its semi-edges, loops and duplicate edges. Now, observe that the quotient pregraph of any cubic $\ell$-circulant nut graph is connected, and thus has a connected underlying graph. Indeed, this is a direct consequence of Lemma \ref{basic_nut_lemma}. Thus, we can obtain all the quotient pregraphs of order $\ell$ by generating all the connected subcubic graphs on $\ell$ vertices and then adding all the combinations of missing semi-edges, loops and duplicate edges until a cubic pregraph is reached. This can be accomplished, e.g., by using the program \texttt{geng} from the package \texttt{nauty} \cite{McKayPip2014}, alongside the \texttt{SageMath} script given in \cite{GitHub}. The script relies on the fact that no vertex can have more than one semi-edge; moreover, triple edges cannot exist whenever $\ell \ge 3$. The numbers of these quotient pregraphs and underlying graphs can be found in Table \ref{sage_info}.

\begin{table}[h!]
\centering
\begin{tabular}{|l|*{7}{r|}}\hline
$\ell$ & 3 & 4 & 5 & 6 & 7 & 8 & 9 \\\hline
$U(\ell)$ & 2 & 6 & 10 & 29 & 64 & 194 & 531 \\\hline
$Q(\ell)$& 4 & 12 & 22 & 68 & 166 & 534 & 1589 \\\hline
\end{tabular}
\caption{$Q(\ell)$ signifies the number of connected cubic quotient pregraphs of order $\ell$ with at most one semi-edge around each vertex, while $U(\ell)$ denotes the number of underlying graphs of these quotient pregraphs, i.e., the number of connected subcubic graphs of order $\ell \ge 3$. Compare with \cite{VanCleemput2014_A} and \cite{VanCleemput2014_B}.}
\label{sage_info}
\end{table}

We now disclose a two-step verification method based on orbit magnitudes that can potentially prove the nonexistence of cubic $\ell$-circulant nut graphs with a particular quotient pregraph $X$ of order $\ell \ge 3$. In the first step, we inspect $\mathcal{N}(X)$ and observe the following.

\begin{proposition}\label{test_1_prop}
    Let $X$ be a connected cubic pregraph of order $\ell \ge 3$ and let $G$ be a simple graph of order $n \in \mathbb{N}$, so that $G$ is the derived graph of the $\mathbb{Z}_{n / \ell}$-voltage pregraph $(X, \varphi)$. If $\mathcal{N}(X)$ contains a nonzero nonfull vector, then $G$ is not a nut graph.
\end{proposition}
\begin{proof}
    Let $u \in \mathcal{N}(X)$ be a nonzero nonfull vector. Now, let $w \in \mathbb{R}^{V(G)}$ be the vector such that $w(x)$ equals the $u$-entry corresponding to the quotient of $x$, for each $x \in V(G)$. In this case, we have that $w$ is a nonzero nonfull vector that belongs to $\mathcal{N}(G)$, hence the graph $G$ cannot be a nut graph.
\end{proof}

By virtue of Proposition \ref{test_1_prop}, if $\dim \mathcal{N}(X) \ge 2$, then $\mathcal{N}(X)$ necessarily contains a nonzero nonfull vector, which means that $X$ cannot give rise to a nut graph. On the other hand, if $\dim \mathcal{N}(X) = 1$, then we compute a basis vector and verify if it has a zero entry. If the said vector is full, then we cannot use Proposition \ref{test_1_prop} to disprove the existence of cubic $\ell$-circulant nut graphs with the given quotient pregraph. Finally, if $\dim \mathcal{N}(X) = 0$, then we again cannot conclude anything via Proposition \ref{test_1_prop}, and thus proceed to the second step. This leads us to the next proposition.

\begin{proposition}\label{test_2_prop}
    Let $X$ be a connected cubic pregraph of order $\ell \ge 3$ for which there is a nut graph of order $n \in \mathbb{N}$ that is the derived graph of a $\mathbb{Z}_{n / \ell}$-voltage pregraph $(X, \varphi)$. Then there exists a matrix $B \in \mathbb{R}^{V(X) \times V(X)}$ such that $|B| = A(X)$ and $\mathcal{N}(B)$ contains a positive vector.
\end{proposition}
\begin{proof}
    Let $G$ be the nut graph of order $n \in \mathbb{N}$ that is the derived graph of $(X, \varphi)$. Moreover, let $u \in \mathbb{R}^{V(G)}$ be a nonzero null space vector of $G$ and let $w \in \mathbb{R}^{V(X)}$ be the positive vector containing the orbit magnitudes with respect to $u$.
    Furthermore, let $\sigma \colon V(G) \to V(X)$ be the function that maps each $V(G)$ vertex to its quotient, so that $w(\sigma(y)) = |u(y)|$ for any $y \in V(G)$.
    
    Let $x \in V(X)$ be a fixed vertex and let $y \in V(G)$ be any vertex such that $\sigma(y) = x$. The local condition in graph $G$ for the vertex $y$ is of the form
    \[
        u(y_1) + u(y_2) + u(y_3) = 0,
    \]
    where $y_1$, $y_2$ and $y_3$ are the neighbors of $y$. This directly leads us to
    \[
        \tau_1 \, w(\sigma(y_1)) + \tau_2 \, w(\sigma(y_2)) + \tau_3 \, w(\sigma(y_3)) = 0,
    \]
    for some choice of $\tau_1, \tau_2, \tau_3 \in \{ -1, +1\}$. It is obvious that the vertices $y_1$, $y_2$ and $y_3$ cannot all reside in the same orbit. Now, if precisely two of them, say $y_1$ and $y_2$, do have the same quotient, then we must have $u(y_1) = u(y_2)$, since otherwise, we would obtain $u(y_3) = 0$, which contradicts the fact that $G$ is a nut graph. Therefore, the vector $w$ is necessarily orthogonal to some row vector $t^{(x)} \in \mathbb{R}^{V(X)}$ such that $|t^{(x)}|$ matches the $x$-row of $A(X)$. By stacking up all such $t^{(x)}$ vectors with respect to $x \in V(X)$, we obtain the desired matrix $B$.
\end{proof}

Note that the matrix $B$ is not necessarily symmetric, hence whenever we refer to its null space, we always consider the right null space. Due to Proposition~\ref{test_2_prop}, we can prove that the given pregraph $X$ cannot give rise to a nut graph if we go through all of the finitely many matrices $B \in \mathbb{R}^{V(X) \times V(X)}$ such that $|B| = A(X)$ and show that none of them contain a positive null space vector. Here, we may assume without loss of generality that the first nonzero entry of each row of $B$ is positive, since changing a row of a matrix to its (additive) inverse does not affect the null space. Finding a positive null space vector gets down to solving a corresponding ILP (integer linear programming)  problem, as shown in \cite{GitHub}. We now give a concrete example to clarify Propositions~\ref{test_1_prop} and~\ref{test_2_prop}.

\begin{example}
Let $X$ be the connected cubic pregraph given in Figure \ref{test_ex_1}. Clearly, we have
\[
    A(X) = \begin{bmatrix}
        0 & 2 & 1\\
        2 & 0 & 1\\
        1 & 1 & 1
    \end{bmatrix} .
\]
It is straightforward to compute that $\mathcal{N}(X)$ is one-dimensional and that it is spanned by the vector $\begin{bmatrix} 1 & 1 & - 2\end{bmatrix}^\intercal$. Thus, Proposition \ref{test_1_prop} fails to guarantee that no graph of order $n \in \mathbb{N}$ derived from a $\mathbb{Z}_{n / 3}$-voltage pregraph $(X, \varphi)$ can be a nut graph.

We proceed to Proposition \ref{test_2_prop} and find the matrix
\[
    B = \begin{bmatrix}
        0 & 2 & -1\\
        2 & 0 & -1\\
        1 & 1 & -1
    \end{bmatrix}
\]
such that $|B| = A(X)$, while $\mathcal{N}(B)$ contains the positive vector $\begin{bmatrix} 1 & 1 & 2 \end{bmatrix}^\intercal$. With this in mind, Proposition \ref{test_2_prop} also fails to yield that no graph $G$ derived from $X$ can be a nut graph. Indeed, this is because there does exist a concrete nut graph $G_0$ derived from $(X, \varphi_0)$ for an according voltage assignment $\varphi_0$, as shown in Figure~\ref{test_ex_2}.

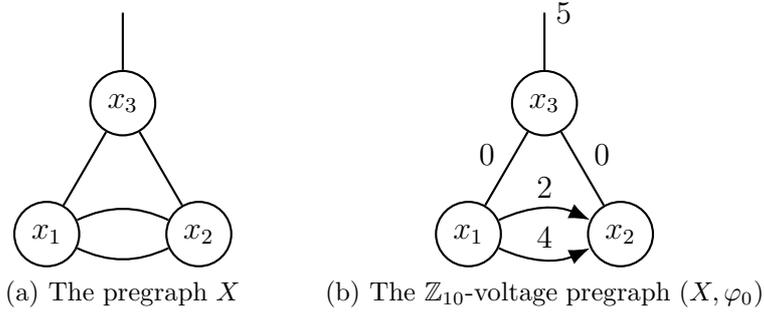
\begin{figure}
    \centering
    \subfloat[The pregraph $X$] {
        \begin{tikzpicture}[scale=0.8]
            \node[state, minimum size=0.75cm, thick] (1) at (1.25, -2.17) {$x_1$};
            \node[state, minimum size=0.75cm, thick] (2) at (2.5, 0) {$x_3$};
            \node[state, minimum size=0.75cm, thick] (3) at (3.75, -2.17) {$x_2$};
    
            \path[thick] (1) edge (2);
            \path[thick] (2) edge (3);
            \draw[thick] (1) edge[bend left=25] (3);
            \draw[thick] (1) edge[bend right=25] (3);
            \draw[thick] (2) to (2.5, 1.5);
        \end{tikzpicture}
        \label{test_ex_1}
    }
    \hspace{0.6cm}
    \subfloat[The $\mathbb{Z}_{10}$-voltage pregraph $(X, \varphi_0)$] { \qquad\qquad
        \begin{tikzpicture}[scale=0.8]
            \node[state, minimum size=0.75cm, thick] (1) at (1.25, -2.17) {$x_1$};
            \node[state, minimum size=0.75cm, thick] (2) at (2.5, 0) {$x_3$};
            \node[state, minimum size=0.75cm, thick] (3) at (3.75, -2.17) {$x_2$};
    
            \path[thick] (1) edge node[pos=0.4, above, xshift=-0.2cm] {$0$} (2);
            \path[thick] (2) edge node[pos=0.6, above, xshift=0.2cm] {$0$} (3);
            \draw[edge, thick] (1) to[bend left=25] node[pos=0.5, above] {$2$} (3);
            \draw[edge, thick] (1) to[bend right=25] node[pos=0.5, above] {$4$} (3);
            \draw[thick] (2) to node[pos=1, right] {$5$} (2.5, 1.5);
        \end{tikzpicture}  \qquad\qquad
        \label{test_ex_2}
    }
    \caption{A $\mathbb{Z}_{10}$-voltage pregraph $(X, \varphi_0)$ that gives rise to the cubic $3$-circulant nut graph $G_0$.}
\end{figure}
\end{example}

By running the \texttt{SageMath} script given in \cite{GitHub}, which relies on Propositions \ref{test_1_prop} and~\ref{test_2_prop}, we now obtain the following result.

\begin{theorem}\label{nonexist}
    There exists no cubic $4$-circulant nut graph and no cubic $5$-circulant nut graph.
\end{theorem}

\section{Cubic 7-circulant nut graphs}\label{sc_7_circ}

Besides proving Theorem \ref{nonexist}, the \texttt{SageMath} script given in \cite{GitHub} is also capable of finding cubic pregraphs of a given order $\ell \ge 3$ that may potentially give rise to nut graphs. Using these computational results, we disclose an infinite family of cubic $7$-circulant nut graphs. To begin, for each $n, \alpha, \beta \in \mathbb{N}$ such that $n \ge 4, \, 2 \mid n$ and $1 \le \alpha, \beta < \frac{n}{2}$, let $G^{(7)}(n; \alpha, \beta)$ denote the $\mathbb{Z}_n$-voltage pregraph given in Figure~\ref{inf_family_1}. We observe that such a voltage graph gives rise to a simple cubic $7$-circulant graph. For convenience, we will denote the vertices of this derived graph via $a_j, b_j, \ldots, g_j, \, j \in \mathbb{Z}_n$ instead of $(a, j), (b, j), \ldots, (g, j), \, j \in \mathbb{Z}_n$. We will now demonstrate that, under certain conditions, $G^{(7)}(n; \alpha, \beta)$ gives rise to a nut graph. Our first step is to prove the following lemma.

\begin{figure}[htb]
    \centering
    \begin{tikzpicture}[scale=0.9]
        \node[state, minimum size=0.90cm, thick] (1) at (0, 0) {$a$};
        \node[state, minimum size=0.90cm, thick] (2) at (2, 0) {$b$};
        \node[state, minimum size=0.90cm, thick] (3) at (4, 0) {$c$};
        \node[state, minimum size=0.90cm, thick] (4) at (5.73, 1) {$d$};
        \node[state, minimum size=0.90cm, thick] (5) at (5.73, -1) {$e$};
        \node[state, minimum size=0.90cm, thick] (6) at (7.46, 0) {$f$};
        \node[state, minimum size=0.90cm, thick] (7) at (9.46, 0) {$g$};

        \draw[thick] (1) to node[pos=0.5, below] {$0$} (2);
        \draw[thick] (2) to node[pos=0.5, below] {$0$} (3);
        \draw[thick] (3) to node[pos=0.5, above, xshift=-0.1cm] {$0$} (4);
        \draw[thick] (3) to node[pos=0.5, below, xshift=-0.1cm] {$0$} (5);
        \draw[thick] (4) to node[pos=0.5, above, xshift=0.1cm] {$0$} (6);
        \draw[thick] (5) to node[pos=0.5, below, xshift=0.1cm] {$0$} (6);
        \draw[thick] (6) to node[pos=0.5, below] {$0$} (7);

        \draw[thick] (2) to node[pos=1, right] {$\frac{n}{2}$} (2, 1.2);
        \draw[thick] (4) to node[pos=1, right] {$\frac{n}{2}$} (5.73, 2.2);
        \draw[thick] (5) to node[pos=1, right] {$\frac{n}{2}$} (5.73, -2.2);

        \draw[edge, thick] (1) to[out=135, in=-135, looseness=6] node[pos=0.1, above] {$\alpha$} (1);
        \draw[edge, thick] (7) to[out=45, in=-45, looseness=6] node[pos=0.1, above] {$\beta$} (7);
    \end{tikzpicture}
    \caption{The $\mathbb{Z}_n$-voltage pregraph $G^{(7)}(n; \alpha, \beta)$ that gives rise to a nut graph whenever $n \ge 4, \, 4 \mid n$ and $\alpha=\beta=1$, or $n \ge 6, \, n \equiv_4 2$ and $\alpha=\beta=2$.}
    \label{inf_family_1}
\end{figure}
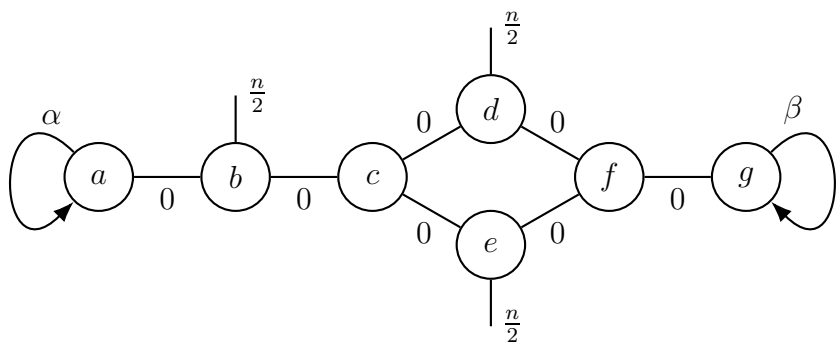

\begin{lemma}\label{circ_7_lemma}
    For any $n \ge 4, \, 2 \mid n$ and $1 \le \alpha, \beta < \frac{n}{2}$, the $\mathbb{Z}_n$-voltage pregraph $G^{(7)}(n; \alpha, \beta)$ gives rise to a nut graph if and only if the $\mathbb{Z}[x]$ polynomial
    \begin{equation}\label{fam_poly1}
        3x^{2\alpha + \beta + \frac{n}{2}} + 3x^{\beta + \frac{n}{2}} + 2x^{2\alpha + 2\beta} + 2x^{2\alpha} - 2x^{\alpha + \beta} + 2x^{2\beta} + 2
    \end{equation}
    contains only $-1$ as a root among all the $n$-th roots of unity.
\end{lemma}
\begin{proof}
    Let $G$ be the graph derived from $G^{(7)}(n; \alpha, \beta)$. From Lemma \ref{local_condition_lemma}, it follows that any vector $u \in \mathbb{R}^{V(G)}$ belongs to $\mathcal{N}(G)$ if and only if it represents a solution to the following system of equations:
    
    \begin{alignat}{2}
        \label{fam_aux1} u\left( a_{j + \alpha} \right) + u\left( a_{j - \alpha }\right) + u\left( b_{j}\right) &= 0 \qquad && (j \in \mathbb{Z}_n) ,\\
        \label{fam_aux2} u\left( a_{j}\right) + u\left( b_{j + \frac{n}{2} }\right) + u\left( c_{j} \right) &= 0 \qquad && (j \in \mathbb{Z}_n) ,\\
        \label{fam_aux3} u\left( b_{j}\right) + u\left( d_{j}\right) + u\left( e_{j} \right) &= 0 \qquad && (j \in \mathbb{Z}_n) ,\\
        \label{fam_aux4} u\left( c_{j}\right) + u\left( d_{j + \frac{n}{2} }\right) + u\left( f_{j} \right) &= 0 \qquad && (j \in \mathbb{Z}_n) ,\\
        \label{fam_aux5} u\left( c_{j}\right) + u\left( e_{j + \frac{n}{2} }\right) + u\left( f_{j} \right) &= 0 \qquad && (j \in \mathbb{Z}_n) ,\\
        \label{fam_aux6} u\left( d_{j}\right) + u\left( e_{j}\right) + u\left( g_{j} \right) &= 0 \qquad && (j \in \mathbb{Z}_n) ,\\
        \label{fam_aux7} u\left( f_{j}\right) + u\left( g_{j + \beta}\right) + u\left( g_{j - \beta} \right) &= 0 \qquad && (j \in \mathbb{Z}_n) .
    \end{alignat}
    
    Equations \eqref{fam_aux3} and \eqref{fam_aux6} together yield $u\left( g_{j} \right) = u\left( b_{j} \right)$, while combining Equations~\eqref{fam_aux4} and \eqref{fam_aux5} gives us
    \begin{alignat}{2}
        \label{fam_aux8} u\left( e_j\right) &= u\left( d_{j}\right) \qquad && (j \in \mathbb{Z}_n) .
    \end{alignat}
    Furthermore, Equation \eqref{fam_aux6} allows us to conclude that
    \begin{alignat}{2}
        \label{fam_aux9} u\left( g_j\right) &= - 2u\left( d_{j} \right) \qquad && (j \in \mathbb{Z}_n) ,
    \end{alignat}
    which then also implies
    \begin{alignat}{2}
        \label{fam_aux10} u\left( b_j\right) &= - 2u\left( d_{j} \right) \qquad && (j \in \mathbb{Z}_n) .
    \end{alignat}
    By plugging in Equation~\eqref{fam_aux9} into Equation~\eqref{fam_aux7}, we obtain
    \begin{alignat}{2}
        \label{fam_aux11} u\left( f_j\right) &= 2u\left( d_{j + \beta}\right) + 2u\left( d_{j - \beta}\right) \qquad && (j \in \mathbb{Z}_n) .
    \end{alignat}
    Now, by plugging in Equation~\eqref{fam_aux11} into Equation~\eqref{fam_aux4}, we see that
    \begin{alignat}{2}
        \label{fam_aux12} u\left( c_j\right) &= -u\left( d_{j + \frac{n}{2}} \right) - 2u\left( d_{j + \beta}\right) - 2u\left( d_{j - \beta}\right) \qquad && (j \in \mathbb{Z}_n).
    \end{alignat}
    By combining Equations \eqref{fam_aux10} and \eqref{fam_aux12} together with Equation~\eqref{fam_aux2}, we are also able to reach
    \begin{alignat}{2}
        \label{fam_aux13} u\left( a_j\right) &= 3u\left( d_{j + \frac{n}{2}} \right) + 2u\left( d_{j + \beta}\right) + 2u\left( d_{j - \beta}\right) \qquad && (j \in \mathbb{Z}_n) .
    \end{alignat}
    Finally, by plugging in Equations \eqref{fam_aux10} and \eqref{fam_aux13} into Equation \eqref{fam_aux1}, we get
    \begin{align}\label{fam_aux14}
        \begin{split}
            3u\left( d_{j + \alpha + \frac{n}{2}} \right) &+ 2u\left( d_{j + \alpha + \beta}\right) + 2u\left( d_{j + \alpha - \beta}\right) + 3u\left( d_{j - \alpha + \frac{n}{2}} \right)\\
            &+ 2u\left( d_{j - \alpha + \beta}\right) + 2u\left( d_{j - \alpha - \beta}\right) - 2u\left( d_j \right) = 0 \qquad (j \in \mathbb{Z}_n) .
        \end{split}
    \end{align}
    It is straightforward to verify that the converse is also true, i.e., that the system of Equations \eqref{fam_aux1}--\eqref{fam_aux7} is actually equivalent to the system of Equations \eqref{fam_aux8}--\eqref{fam_aux14}.

    If $D$ consists of all the vertices from the $d$-orbit of $G$, then for any vector $u \in \mathbb{R}^D$ satisfying Equation~\eqref{fam_aux14}, there exists a unique extension to an $\mathbb{R}^{V(G)}$ vector which belongs to $\mathcal{N}(G)$. In fact, $\mathcal{N}(G)$ is composed of precisely all the vectors of the form $Wu$, where $u \in \mathbb{R}^D$ is a solution to Equation~\eqref{fam_aux14}, and $W \in \mathbb{R}^{V(G) \times D}$ is a matrix such that:
    \begin{enumerate}[label=\textbf{(\roman*)}]
        \item its $D \times D$ submatrix is the identity matrix;
        \item the elements of its $\left( V(G) \setminus D \right) \times D$ submatrix are determined by the Equations \eqref{fam_aux8}--\eqref{fam_aux13}.
    \end{enumerate}
    Since $W$ has full column rank, it can be observed that $\dim \mathcal{N}(G)$ coincides with the dimension of the solution set of Equation~\eqref{fam_aux14}. This solution set can be regarded as the null space of an according circulant matrix $C \in \mathbb{R}^{\mathbb{Z}_n \times \mathbb{Z}_n}$ whose spectrum is given by
    \begin{equation}\label{fam_aux15}
        3 \zeta^{\alpha + \frac{n}{2}} + 2 \zeta^{\alpha + \beta} + 2 \zeta^{\alpha - \beta} + 3\zeta^{-\alpha + \frac{n}{2}} + 2\zeta^{-\alpha+\beta} + 2\zeta^{-\alpha-\beta} - 2 ,
    \end{equation}
    as $\zeta \in \mathbb{C}$ ranges over the $n$-th roots of unity.

    Due to $\dim \mathcal{N}(G) = \dim \mathcal{N}(C)$, we conclude that if $G$ is a nut graph, then Expression \eqref{fam_aux15} must yield zero for precisely one $n$-th root of unity $\zeta$. It is obvious that $\zeta = 1$ cannot be a root of Expression \eqref{fam_aux15}, and if Expression \eqref{fam_aux15} had a nonreal $n$-th root of unity $\zeta_0$ as a root, then it would also contain $\overline{\zeta_0} \neq \zeta_0$ as a root, which is impossible. Therefore, in this scenario we get that Expression \eqref{fam_aux15} contains only $-1$ as a root among all the $n$-th roots of unity. By multiplying the entire expression with $\zeta^{\alpha + \beta}$, we notice that Polynomial \eqref{fam_poly1} contains only $-1$ as a root among all the $n$-th roots of unity.
    
    Now, suppose that Polynomial \eqref{fam_poly1} does contain only $-1$ as a root among all the $n$-th roots of unity. In this case, we get $\dim \mathcal{N}(G) = \dim \mathcal{N}(C) = 1$. Moreover, $\mathcal{N}(C)$ is spanned by the vector $\begin{bmatrix} 1 & -1 & 1 & -1 & \cdots & 1 & -1 \end{bmatrix}^\intercal$, hence Equation~\eqref{fam_aux14} contains a solution $w_0 \in \mathbb{R}^D$ such that
    \[
        w_0 \left( d_j \right) = (-1)^j \qquad (0 \le j < n) .
    \]
    It then follows that $u_0 = W w_0 \in \mathcal{N}(G)$. Equations \eqref{fam_aux8}--\eqref{fam_aux13} make it easy to check that $u_0$ is a full vector, since $\left| w_0 \left( d_j \right) \right|$ is constant for $j \in \mathbb{Z}_n$, while
    \[
        u_0 \left( f_j\right) = 2w_0\left( d_{j + \beta}\right) + 2w_0 \left( d_{j - \beta}\right) = 4w_0\left( d_{j + \beta}\right) \neq 0
    \]
    for any $j \in \mathbb{Z}_n$. Therefore, graph $G$ is a nut graph.
\end{proof}

We now disclose the main proposition of this section that demonstrates how for any $n \ge 4, \, 2 \mid n$, the parameters $\alpha, \beta$ can be configured so that the voltage pregraph $G^{(7)}(n; \alpha, \beta)$ gives rise to a nut graph.

\begin{proposition}\label{circ_7_prop}
    The graph derived from $G^{(7)}(n; \alpha, \beta)$ is a nut graph whenever one of the following two conditions is satisfied:
    \begin{enumerate}[label=\textbf{(\roman*)}]
        \item $n \ge 4, \, 4 \mid n$ and $\alpha = \beta = 1$;
        \item $n \ge 6, \, n \equiv_4 2$ and $\alpha = \beta = 2$.
    \end{enumerate}
\end{proposition}
\begin{proof}
    First, suppose that $n \ge 4, \, 4 \mid n$ and $\alpha = \beta = 1$. In this case, Polynomial~\eqref{fam_poly1} transforms into
    \begin{equation}\label{fam_aux16}
        3x^{3 + \frac{n}{2}} + 3x^{1 + \frac{n}{2}} + 2x^4 + 2x^2 + 2 .
    \end{equation}
    For any $n$-th root of unity $\zeta \in \mathbb{C}$ such that $\zeta^\frac{n}{2} = -1$, we have that $\zeta$ is a root of Polynomial \eqref{fam_aux16} if and only if it is a root of
    \[
        2x^4 - 3x^3 + 2x^2 - 3x + 2 = (x-1)^2 \, (2x^2 + x + 2) .
    \]
    The $2x^2 + x + 2$ polynomial is obviously not divisible by any cyclotomic polynomial, hence it cannot contain any root of unity among its roots, while $\zeta_0 = 1$ does not satisfy $\zeta_0^\frac{n}{2} = -1$. Thus, Polynomial \eqref{fam_aux16} contains no roots among the $n$-th roots of unity $\zeta$ that satisfy $\zeta^\frac{n}{2} = -1$. Similarly, for any $\zeta \in \mathbb{C}$ such that $\zeta^\frac{n}{2} = 1$, it is clear that $\zeta$ is a root of Polynomial \eqref{fam_aux16} if and only if it is a root of
    \[
        2x^4 + 3x^3 + 2x^2 + 3x + 2 = (x+1)^2 \, (2x^2 - x + 2) .
    \]
    Since $2x^2 + x + 2$ is not divisible by any cyclotomic polynomial, from $(-1)^\frac{n}{2} = 1$ we obtain that Polynomial \eqref{fam_aux16} contains only $-1$ as a root among all the $n$-th roots of unity. Thus, due to Lemma~\ref{circ_7_lemma}, we conclude that $G^{(7)}(n; \alpha, \beta)$ gives rise to a nut graph.

    Now, suppose that $n \ge 6, \, n \equiv_4 2$ and $\alpha = \beta = 2$. In this scenario, Polynomial~\eqref{fam_poly1} transforms into
    \begin{equation}\label{fam_aux17}
        3x^{6 + \frac{n}{2}} + 3x^{2 + \frac{n}{2}} + 2x^8 + 2x^4 + 2 .
    \end{equation}
    For an arbitrary $n$-th root of unity $\zeta \in \mathbb{C}$ satisfying $\zeta^\frac{n}{2} = 1$, we see that $\zeta$ is a root of Polynomial \eqref{fam_aux17} if and only if it is a root of
    \[
        2x^8 + 3x^6 + 2x^4 + 3x^2 + 2 = (x^2 + 1)^2 \, (2x^4 - x^2 + 2) .
    \]
    However, it is easy to verify that $2x^4 - x^2 + 2$ is not divisible by any cyclotomic polynomial, which means that this polynomial cannot contain a root of unity among its roots. Since $n \equiv_4 2$, the complex numbers $i$ and $-i$ do not represent an $n$-th root of unity, hence Polynomial \eqref{fam_aux17} has no root $\zeta$ that is an $n$-th root of unity for which $\zeta^\frac{n}{2} = 1$. Finally, for any $\zeta \in \mathbb{C}, \, \zeta^n = 1$ such that $\zeta^\frac{n}{2} = -1$, we have that $\zeta$ is a root of Polynomial \eqref{fam_aux17} if and only if it is a root of
    \[
        2x^8 - 3x^6 + 2x^4 - 3x^2 + 2 = (x-1)^2 \, (x+1)^2 \, (2x^4 + x^2 + 2) .
    \]
    The $2x^4 + x^2 + 2$ polynomial has no roots of unity among its roots since it is obviously not divisible by any cyclotomic polynomial. For $\zeta_0 = 1$, we have $\zeta_0^\frac{n}{2} = 1$, while for $\zeta_0 = -1$, it follows that $\zeta_0^\frac{n}{2} = -1$ since $2 \nmid \frac{n}{2}$. We may conclude that Polynomial \eqref{fam_aux17} does contain only $-1$ as a root among all the $n$-th roots of unity. Therefore, Lemma \ref{circ_7_lemma} guarantees that the graph derived from $G^{(7)}(n; \alpha, \beta)$ is a nut graph.
\end{proof}

Proposition \ref{circ_7_prop} implies that for any $n \ge 4, \, 2 \mid n$, there exists a cubic $7$-circulant nut graph of order $7n$. As a direct consequence, we obtain that there are infinitely many cubic $7$-circulant nut graphs.

\begin{corollary}
    There exist infinitely many cubic $7$-circulant nut graphs.
\end{corollary}

\section{Cubic 11-circulant nut graphs}\label{sc_11_circ}

We apply a similar strategy as in Section \ref{sc_7_circ} to show that there are also infinitely many cubic $11$-circulant nut graphs. First, for each $n, \alpha, \beta \in \mathbb{N}$ such that $2 \mid n$ and $1 \le \alpha, \beta < n$, let $G^{(11)}(n; \alpha, \beta)$ denote the $\mathbb{Z}_n$-voltage pregraph shown in Figure~\ref{inf_family_2}. It is obvious that the derived graph of each $G^{(11)}(n; \alpha, \beta)$ is a simple cubic $11$-circulant graph. For convenience, we will label its vertices as $a'_j, a''_j, b'_j, b''_j, \ldots, f_j, \, j \in \mathbb{Z}_n$ instead of $(a', j), (a'', j), \ldots, (f, j), \, j \in \mathbb{Z}_n$. We proceed by giving an analog to Lemma \ref{circ_7_lemma}.

\begin{figure}[htb]
    \centering
    \begin{tikzpicture}[scale=0.9]
        \node[state, minimum size=0.90cm, thick] (1) at (0, 3.236) {$c'$};
        \node[state, minimum size=0.90cm, thick] (2) at (1.902, 2.618) {$d'$};
        \node[state, minimum size=0.90cm, thick] (3) at (3.078, 1) {$e'$};
        \node[state, minimum size=0.90cm, thick] (4) at (3.078, -1) {$e''$};
        \node[state, minimum size=0.90cm, thick] (5) at (1.902, -2.618) {$d''$};
        \node[state, minimum size=0.90cm, thick] (6) at (0, -3.236) {$c''$};
        \node[state, minimum size=0.90cm, thick] (7) at (-1.902, -2.618) {$b''$};
        \node[state, minimum size=0.90cm, thick] (8) at (-3.078, -1) {$a''$};
        \node[state, minimum size=0.90cm, thick] (9) at (-3.078, 1) {$a'$};
        \node[state, minimum size=0.90cm, thick] (10) at (-1.902, 2.618) {$b'$};
        \node[state, minimum size=0.90cm, thick] (11) at (4.810, 0) {$f$};

        \draw[thick] (1) to node[pos=0.5, below] {$0$} (2);
        \draw[thick] (2) to node[pos=0.4, below, xshift=-0.1cm] {$0$} (3);
        \draw[thick] (3) to node[pos=0.5, left] {$0$} (4);
        \draw[thick] (4) to node[pos=0.6, above, xshift=-0.1cm] {$0$} (5);
        \draw[thick] (5) to node[pos=0.5, above] {$0$} (6);
        \draw[thick] (6) to node[pos=0.5, above] {$0$} (7);
        
        \draw[edge, thick] (8) to[bend right=30] node[pos=0.45, left, xshift=-0.05cm] {$\beta$} (7);
        \draw[thick] (8) to[bend left=30] node[pos=0.4, right, xshift=0.05cm] {$0$} (7);
        
        \draw[thick] (8) to node[pos=0.5, right] {$0$} (9);
        
        \draw[edge, thick] (9) to[bend left=30] node[pos=0.45, left, xshift=-0.05cm] {$\alpha$} (10);
        \draw[thick] (9) to[bend right=30] node[pos=0.4, right, xshift=0.05cm] {$0$} (10);
        
        \draw[thick] (10) to node[pos=0.5, below] {$0$} (1);
        \draw[thick] (3) to node[pos=0.55, above] {$0$} (11);
        \draw[thick] (4) to node[pos=0.55, below] {$0$} (11);

        \draw[thick] (1) to node[pos=1, right] {$\frac{n}{2}$} (0, 4.436);
        \draw[thick] (6) to node[pos=1, right] {$\frac{n}{2}$} (0, -4.436);
        \draw[thick] (2) to node[pos=1, right] {$\frac{n}{2}$} (2.607, 3.589);
        \draw[thick] (5) to node[pos=1, right] {$\frac{n}{2}$} (2.607, -3.589);
        \draw[thick] (11) to node[pos=1, right] {$\frac{n}{2}$} (6.010, 0);
    \end{tikzpicture}
    \caption{The $\mathbb{Z}_n$-voltage pregraph $G^{(11)}(n; \alpha, \beta)$ that gives rise to a nut graph whenever $n \ge 6, \, n \equiv_4 2$ and $\alpha=\beta=2$.}
    \label{inf_family_2}
\end{figure}
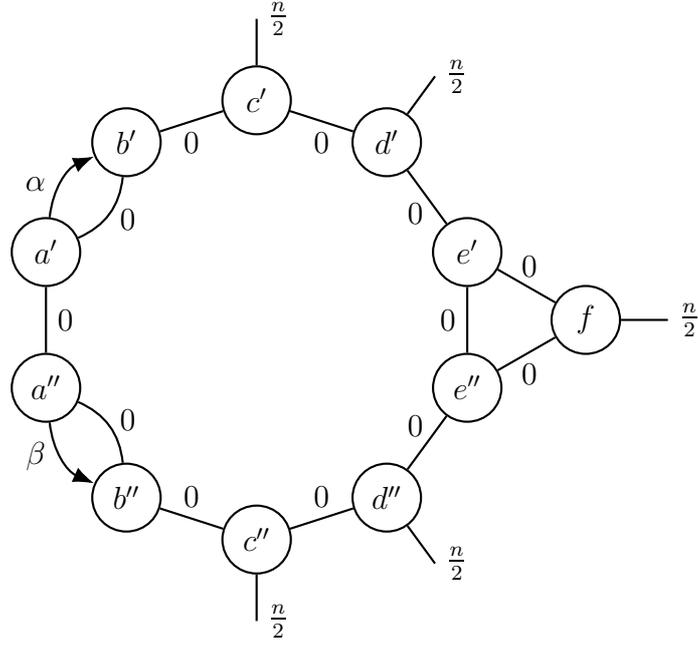

\begin{lemma}\label{circ_11_lemma}
    For any $2 \mid n$ and $1 \le \alpha, \beta < n$, the $\mathbb{Z}_n$-voltage pregraph $G^{(11)}(n; \alpha, \beta)$ gives rise to a nut graph if and only if $2 \mid \alpha, \beta$ and $n \equiv_4 2$, while the $\mathbb{Z}[x]$ polynomial
    \begin{align}
        \label{fam_poly2}
        \begin{split}
            x^{2\alpha + 2\beta} + x^{2\alpha + \beta} &+ x^{\alpha + 2\beta} + x^{\alpha} + x^{\beta} + 1\\
            &+ x^{2\alpha + \beta + \frac{n}{2}} + x^{\alpha + 2\beta + \frac{n}{2}} + x^{\alpha + \frac{n}{2}} + x^{\beta + \frac{n}{2}} + x^{2\alpha + \frac{n}{2}} + x^{2\beta + \frac{n}{2}}
        \end{split}
    \end{align}
    contains only $-1$ as a root among all the $n$-th roots of unity.
\end{lemma}
\begin{proof}
    To begin, let $G$ be the graph derived from $G^{(11)}(n; \alpha, \beta)$. Lemma \ref{local_condition_lemma} tells us that any given vector $u \in \mathbb{R}^{V(G)}$ belongs to $\mathcal{N}(G)$ if and only if it represents a solution to the following system of equations:
    \allowdisplaybreaks
    \begin{alignat}{2}
        \label{fam_aux18} u\left( a''_j \right) + u\left( b'_j \right) + u\left( b'_{j + \alpha}\right) &= 0 \qquad && (j \in \mathbb{Z}_n) ,\\
        \label{fam_aux19} u\left( a'_j \right) + u\left( b''_j \right) + u\left( b''_{j + \beta}\right) &= 0 \qquad && (j \in \mathbb{Z}_n) ,\\
        \label{fam_aux20} u\left( a'_j \right) + u\left( a'_{j - \alpha}\right) + u\left( c'_j \right) &= 0 \qquad && (j \in \mathbb{Z}_n) ,\\
        \label{fam_aux21} u\left( a''_j \right) + u\left( a''_{j - \beta}\right) + u\left( c''_j \right) &= 0 \qquad && (j \in \mathbb{Z}_n) ,\\
        \label{fam_aux22} u\left( b'_j \right) + u\left( c'_{j + \frac{n}{2} }\right) + u\left( d'_j \right) &= 0 \qquad && (j \in \mathbb{Z}_n) ,\\
        \label{fam_aux23} u\left( b''_j \right) + u\left( c''_{j + \frac{n}{2} }\right) + u\left( d''_j \right) &= 0 \qquad && (j \in \mathbb{Z}_n) ,\\
        \label{fam_aux24} u\left( c'_j \right) + u\left( d'_{j + \frac{n}{2} }\right) + u\left( e'_j \right) &= 0 \qquad && (j \in \mathbb{Z}_n) ,\\
        \label{fam_aux25} u\left( c''_j \right) + u\left( d''_{j + \frac{n}{2} }\right) + u\left( e''_j \right) &= 0 \qquad && (j \in \mathbb{Z}_n) ,\\
        \label{fam_aux26} u\left( d'_j \right) + u\left( e''_j \right) + u\left( f_j \right) &= 0 \qquad && (j \in \mathbb{Z}_n) ,\\
        \label{fam_aux27} u\left( d''_j \right) + u\left( e'_j \right) + u\left( f_j \right) &= 0 \qquad && (j \in \mathbb{Z}_n) ,\\
        \label{fam_aux28} u\left( e'_j \right) + u\left( e''_j \right) + u\left( f_{j + \frac{n}{2}} \right) &= 0 \qquad && (j \in \mathbb{Z}_n) .
    \end{alignat}
    
    Equation~\eqref{fam_aux28} gives us
    \begin{equation}
        \label{fam_aux29} u\left( f_j \right) = -u\left( e'_{j + \frac{n}{2}} \right) - u\left( e''_{j + \frac{n}{2}} \right) \qquad (j \in \mathbb{Z}_n) .
    \end{equation}
    By plugging in Equation~\eqref{fam_aux29} into Equations~\eqref{fam_aux26} and \eqref{fam_aux27}, we are able to reach
    \begin{alignat}{2}
        \label{fam_aux30} u\left( d'_j \right) &= u\left( e'_{j + \frac{n}{2}} \right) + u\left( e''_{j + \frac{n}{2}} \right) - u\left( e''_j \right) \qquad && (j \in \mathbb{Z}_n),\\
        \label{fam_aux31} u\left( d''_j \right) &= u\left( e'_{j + \frac{n}{2}} \right) + u\left( e''_{j + \frac{n}{2}} \right) - u\left( e'_j \right) \qquad && (j \in \mathbb{Z}_n) .
    \end{alignat}
    Furthermore, by plugging in Equation \eqref{fam_aux30} into Equation \eqref{fam_aux24} and Equation \eqref{fam_aux31} into Equation \eqref{fam_aux25}, we get
    \begin{alignat}{2}
        \label{fam_aux32} u\left( c'_j \right) &= -2u\left( e'_j \right) -u\left( e''_j \right) + u\left( e''_{j + \frac{n}{2}} \right) \qquad && (j \in \mathbb{Z}_n),\\
        \label{fam_aux33} u\left( c''_j \right) &= -u\left( e'_j \right) -2u\left( e''_j \right) + u\left( e'_{j + \frac{n}{2}} \right) \qquad && (j \in \mathbb{Z}_n) .
    \end{alignat}
    We plug in Equations \eqref{fam_aux30} and \eqref{fam_aux32} into Equation \eqref{fam_aux22} and Equations~\eqref{fam_aux31} and \eqref{fam_aux33} into Equation \eqref{fam_aux23} to obtain
    \begin{alignat}{2}
        \label{fam_aux34} u\left( b'_j \right) &= u\left( e'_{j+\frac{n}{2}} \right) \qquad && (j \in \mathbb{Z}_n),\\
        \label{fam_aux35} u\left( b''_j \right) &= u\left( e''_{j+\frac{n}{2}} \right) \qquad && (j \in \mathbb{Z}_n) .
    \end{alignat}
    We are now able to plug in Equation \eqref{fam_aux35} into Equation \eqref{fam_aux19} and Equation \eqref{fam_aux34} into Equation \eqref{fam_aux18} to conclude that
    \begin{alignat}{2}
        \label{fam_aux36} u\left( a'_j \right) &= -u\left( e''_{j+\frac{n}{2}} \right) - u\left( e''_{j+\beta + \frac{n}{2}} \right) \qquad && (j \in \mathbb{Z}_n),\\
        \label{fam_aux37} u\left( a''_j \right) &= -u\left( e'_{j+\frac{n}{2}} \right) - u\left( e'_{j+\alpha+\frac{n}{2}} \right) \qquad && (j \in \mathbb{Z}_n) .
    \end{alignat}

    By plugging in Equations \eqref{fam_aux32} and \eqref{fam_aux36} into Equation \eqref{fam_aux20} and Equations \eqref{fam_aux33} and \eqref{fam_aux37} into Equation \eqref{fam_aux21}, a simple computation yields
    \begin{alignat}{2}
        \label{fam_aux38} 2u\left( e'_j \right) + u\left( e''_j \right) + u\left( e''_{j + \beta + \frac{n}{2}} \right) + u\left( e''_{j - \alpha + \frac{n}{2}} \right) + u\left( e''_{j + \beta - \alpha + \frac{n}{2}} \right) &= 0,\\
        \label{fam_aux39} 2u\left( e''_j \right) + u\left( e'_j \right) + u\left( e'_{j + \alpha + \frac{n}{2}} \right) + u\left( e'_{j - \beta + \frac{n}{2}} \right) + u\left( e'_{j + \alpha - \beta + \frac{n}{2}} \right) &= 0,
    \end{alignat}
    for each $j \in \mathbb{Z}_n$. Equation \eqref{fam_aux38} can now be used to express $u\left( e'_j \right)$ in terms of the $u\left( e''_j \right)$ values, thus giving
    \begin{equation}
        \label{fam_aux40} u\left( e'_j \right) = -\frac{1}{2} \left( u\left( e''_j \right) + u\left( e''_{j + \beta + \frac{n}{2}} \right) + u\left( e''_{j - \alpha + \frac{n}{2}} \right) + u\left( e''_{j + \beta - \alpha + \frac{n}{2}} \right) \right) .
    \end{equation}
    Finally, by plugging in Equation~\eqref{fam_aux40} into Equation~\eqref{fam_aux39}, we obtain
    \begin{align}
        \label{fam_aux41}\begin{split}
            u\left(e''_{j+\alpha+\beta}\right) &+ u\left(e''_{j+\alpha}\right) + u\left(e''_{j+\beta}\right) + u\left(e''_{j-\alpha}\right) + u\left(e''_{j-\beta}\right) + u\left(e''_{j-\alpha-\beta}\right)\\
            &+ u\left(e''_{j+\alpha+\frac{n}{2}}\right) + u\left(e''_{j+\beta+\frac{n}{2}}\right) + u\left(e''_{j-\alpha+\frac{n}{2}}\right) + u\left(e''_{j-\beta+\frac{n}{2}}\right)\\
            &+ u\left(e''_{j+\alpha-\beta+\frac{n}{2}}\right) + u\left(e''_{j-\alpha+\beta+\frac{n}{2}}\right) = 0 \qquad (j \in \mathbb{Z}_n) .
        \end{split}
    \end{align}
    It is not difficult to establish that the system of Equations \eqref{fam_aux18}--\eqref{fam_aux28} is equivalent to the system of Equations~\eqref{fam_aux29}--\eqref{fam_aux37}, \eqref{fam_aux40} and \eqref{fam_aux41}. Therefore, it is possible to apply the same strategy as in Lemma \ref{circ_7_lemma} to observe that $\dim \mathcal{N}(G)$ equals the dimension of the solution set of Equation~\eqref{fam_aux41}. This solution set represents the null space of a circulant matrix $C \in \mathbb{R}^{\mathbb{Z}_n \times \mathbb{Z}_n}$ whose spectrum is given by
    \begin{align}
        \label{fam_aux42}
        \begin{split}
            \zeta^{\alpha + \beta} + \zeta^{\alpha} &+ \zeta^{\beta} + \zeta^{-\alpha} + \zeta^{-\beta} + \zeta^{-\alpha-\beta}\\
            &+ \zeta^{\alpha + \frac{n}{2}} + \zeta^{\beta + \frac{n}{2}} + \zeta^{-\alpha + \frac{n}{2}} + \zeta^{-\beta + \frac{n}{2}} + \zeta^{\alpha - \beta + \frac{n}{2}} + \zeta^{-\alpha + \beta + \frac{n}{2}},
        \end{split}
    \end{align}
    as $\zeta \in \mathbb{C}$ ranges over the $n$-th roots of unity.

    Now, if graph $G$ is a nut graph, then $\dim \mathcal{N}(G) = \dim \mathcal{N}(C)$ tells us that Expression~\eqref{fam_aux42} gives zero for precisely one $n$-th root of unity $\zeta$. In an analogous manner as in Lemma \ref{circ_7_lemma}, we see that this $n$-th root of unity must be precisely $-1$. By multiplying Expression \eqref{fam_aux42} with $\zeta^{\alpha + \beta}$, we observe that Polynomial \eqref{fam_poly2} contains only $-1$ as a root among all the $n$-th roots of unity. Moreover, $\mathcal{N}(C)$ is spanned by the vector $\begin{bmatrix} 1 & -1 & 1 & -1 & \cdots & 1 & -1 \end{bmatrix}^\intercal$. From here, Equations~\eqref{fam_aux36} and \eqref{fam_aux37} imply that $2 \mid \alpha, \beta$ must be true, since otherwise, each solution $u \in \mathbb{R}^{V(G)}$ would satisfy $u\left( a_j' \right) = 0$ or $u\left( a_j'' \right) = 0$, both of which would lead to a contradiction. Furthermore, $4 \mid n$ cannot be true because in this scenario, Equations \eqref{fam_aux38} and \eqref{fam_aux39} would get down to
    \[
        2u\left( e'_j \right) + 4u\left( e''_j \right) = 0 \qquad \mbox{and} \qquad 2u\left( e''_j \right) + 4u\left( e'_j \right) = 0,
    \]
    thus yielding $u\left( e'_j \right) = u\left( e''_j \right) = 0$. Therefore, $n \equiv_4 2$ holds.
    
    On the other hand, if we suppose that $2 \mid \alpha, \beta$ and $n \equiv_4 2$, while Polynomial~\eqref{fam_poly2} contains only $-1$ as a root among all the $n$-th roots of unity, it follows that $\dim \mathcal{N}(G) = \dim \mathcal{N}(C) = 1$. Moreover, we have that $\mathcal{N}(C)$ is spanned by the vector $\begin{bmatrix} 1 & -1 & 1 & -1 & \cdots & 1 & -1 \end{bmatrix}^\intercal$, hence if we let $E''$ comprise the vertices from the $e''$-orbit of $G$, we may conclude that Equation \eqref{fam_aux41} contains a solution $w_0 \in \mathbb{R}^{E''}$ such that
    \[
        w_0\left( e''_j \right) = (-1)^j \qquad (0 \le j < n).
    \]
    Equations~\eqref{fam_aux29}--\eqref{fam_aux37} and \eqref{fam_aux40} can now be used to extend $w_0$ to vector $u_0 \in \mathcal{N}(G)$ that is full. Indeed, Equation~\eqref{fam_aux40} yields $u_0\left( e'_j \right) = u_0\left( e''_j \right)$, thus making it trivial to observe that $u_0$ is a full vector by applying the remaining equations. We conclude that graph $G$ is a nut graph.
\end{proof}

Here is the main result of this section.

\begin{proposition}\label{circ_11_prop}
    The graph derived from $G^{(11)}(n; \alpha, \beta)$ is a nut graph whenever $n \ge 6, \, n \equiv_4 2$ and $\alpha = \beta = 2$.
\end{proposition}
\begin{proof}
    Due to Lemma \ref{circ_11_lemma}, in order to show that $G^{(11)}(n; 2, 2)$ gives rise to a nut graph whenever $n \ge 6, \, n \equiv_4 2$, it is sufficient to prove that Polynomial~\eqref{fam_poly2} contains only $-1$ as a root among all the $n$-th roots of unity. This polynomial transforms into
    \[
        x^8 + 2x^6 + 2x^2 + 1 + 2x^{6 + \frac{n}{2}} + 2x^{4 + \frac{n}{2}} + 2x^{2 + \frac{n}{2}} .
    \]
    Therefore, for any $n$-th root of unity $\zeta \in \mathbb{C}$ such that $\zeta^\frac{n}{2} = 1$, it follows that $\zeta$ is a root of Polynomial \eqref{fam_poly2} if and only if it is a root of
    \[
        x^8 + 4x^6 + 2x^4 + 4x^2 + 1 = (x^4 + 1)(x^4 + 4x^2 + 1) .
    \]
    It is straightforward to verify that the $x^4 + 4x^2 + 1$ polynomial is not divisible by any cyclotomic polynomial, which means that it cannot contain any root of unity among its roots. Besides that, the primitive eighth roots of unity do not represent an $n$-th root of unity due to $n \equiv_4 2$, hence Polynomial \eqref{fam_poly2} has no roots among the $n$-th roots of unity $\zeta$ for which $\zeta^\frac{n}{2} = 1$.

    On the other hand, for any $n$-th root of unity $\zeta$ that satisfies $\zeta^\frac{n}{2} = -1$, we obtain that $\zeta$ is a root of Polynomial \eqref{fam_poly2} if and only if it is a root of
    \[
        x^8 - 2x^4 + 1 = (x-1)^2 \, (x+1)^2 \, (x^2 + 1)^2 .
    \]
    The complex numbers $i$ and $-i$ do not represent an $n$-th root of unity since $n \equiv_4 2$, while for $\zeta_0 = 1$, we have $\zeta_0^\frac{n}{2} = 1$. For $\zeta_0 = -1$, we have $\zeta_0^\frac{n}{2} = -1$ due to $n \equiv_4 2$, hence Polynomial \eqref{fam_poly2} contains only $-1$ as a root among all the $n$-th roots of unity.
\end{proof}

As an immediate consequence of Proposition \ref{circ_11_prop}, we may conclude that for any $n \ge 6, \, n \equiv_4 2$, there exists a cubic $11$-circulant nut graph of order $11n$. This implies that there are infinitely many cubic $11$-circulant nut graphs.

\begin{corollary}
    There exist infinitely many cubic $11$-circulant nut graphs.
\end{corollary}

\section{Pre-subdivision construction}\label{sc_subdiv}

Our next step is to show how, under certain conditions, a given cubic $\ell$-circulant nut graph can be extended to a cubic $(\ell+3)$-circulant nut graph. The following lemma plays a central role in finalizing the proof of Theorem \ref{main_th}.

\begin{figure}[htb]
    \centering
    \subfloat[$\mathbb{Z}_n$-voltage pregraph $(X, \varphi)$] {
        \centering
        \begin{tikzpicture}
            \node[state, minimum size=0.75cm, thick] (1) at (0, 0) {$a$};
            \node[state, minimum size=0.75cm, thick] (2) at (2.5, 0) {$b$};
    
            \draw[edge, thick] (1) to node[pos=0.5, below] {$\gamma$} (2);
    
            \draw[dashed, thick] (-0.8, -0.8) -- (0.8, -0.8) -- (0.8, 0.8) -- (1.7, 0.8) -- (1.7, -0.8) -- (3.3, -0.8) -- (3.3, 1.5) -- (-0.8, 1.5) -- (-0.8, -0.8);

            \node (3) at (-2, 0) {};
            \node (4) at (4.5, 0) {};
        \end{tikzpicture}
        \label{presub1}
    }\\
    \subfloat[$\mathbb{Z}_n$-voltage pregraph $(X_1, \varphi_1)$] {
        \begin{tikzpicture}
            \node[state, minimum size=0.75cm, thick] (1) at (0, 0) {$a$};
            \node[state, minimum size=0.75cm, thick] (2) at (2.5, 0) {$c$};
            \node[state, minimum size=0.75cm, thick] (3) at (5.0, 0) {$d$};
            \node[state, minimum size=0.75cm, thick] (4) at (7.5, 0) {$e$};
            \node[state, minimum size=0.75cm, thick] (5) at (10.0, 0) {$b$};
    
            \draw[edge, thick] (1) to node[pos=0.5, below] {$\gamma$} (2);
            \draw[thick] (3) to node[pos=0.5, below] {$\frac{n}{2}$} (2);
            \draw[edge, thick] (4) to node[pos=0.5, below] {$\gamma$} (3);
            \draw[edge, thick] (4) to node[pos=0.5, below] {$\gamma$} (5);

            \draw[thick] (2) to node[pos=1, right] {$\frac{n}{2}$} (2.5, -1);
            \draw[thick] (3) to node[pos=1, right] {$\frac{n}{2}$} (5.0, -1);
            \draw[thick] (4) to node[pos=1, right] {$\frac{n}{2}$} (7.5, -1);
    
            \draw[dashed, thick] (-0.8, -0.8) -- (0.8, -0.8) -- (0.8, 0.8) -- (9.2, 0.8) -- (9.2, -0.8) -- (10.8, -0.8) -- (10.8, 1.5) -- (-0.8, 1.5) -- (-0.8, -0.8);
        \end{tikzpicture}
        \label{presub2}
    }
    \caption{The $\mathbb{Z}_n$-voltage pregraphs $(X, \varphi)$ and $(X_1, \varphi_1)$ from Lemma \ref{presub_lemma}.}
\end{figure}
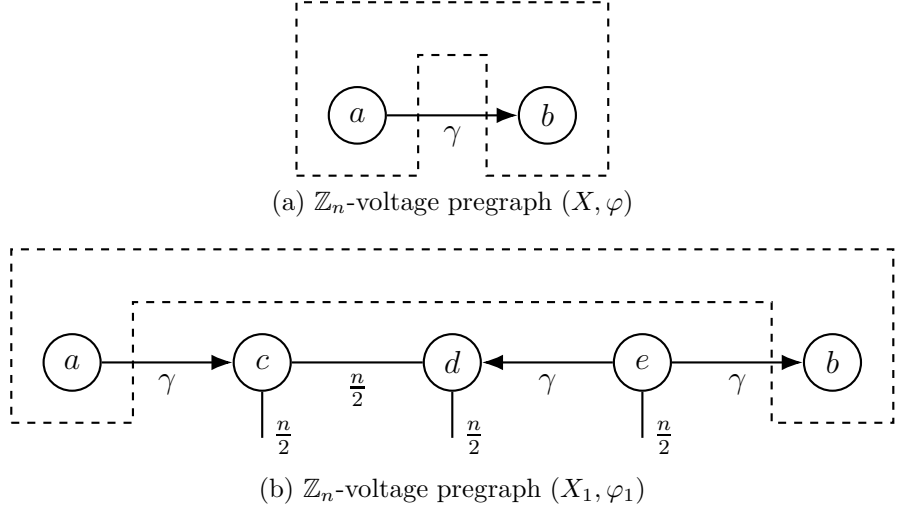

\begin{lemma}\label{presub_lemma}
    For an even $n$, let $X = (D, V; \beg, \inv)$ be a pregraph so that $(X, \varphi)$ is a $\mathbb{Z}_n$-voltage pregraph. Let $a, b \in V$ be two distinct vertices that are connected by an $(a, b)$-dart whose assigned voltage is $\gamma \in \mathbb{Z}_n$, as shown in Figure~\ref{presub1}. Furthermore, let $(X_1, \varphi_1)$ be the $\mathbb{Z}_n$-voltage pregraph that is obtained from $(X, \varphi)$ by removing the edge that contains this dart and adding three new vertices $c, d, e$, alongside certain new darts with voltages assigned as in Figure~\ref{presub2}. If $(X, \varphi)$ gives rise to a nut graph such that the $a$- and $b$-orbits are of different magnitudes, then $(X_1, \varphi_1)$ gives rise to a nut graph.
\end{lemma}
\begin{proof}
    Let $G$ and $G_1$ be the graphs derived from $(X, \varphi)$ and $(X_1, \varphi_1)$, respectively. Due to Lemma \ref{local_condition_lemma}, we know that each vector $u \in \mathbb{R}^{V(G_1)}$ belongs to $\mathcal{N}(G_1)$ if and only if it represents a solution to the next system of equations:
    \begin{alignat}{2}
        \label{presub_aux1} u\left( a_{j - \gamma} \right) + u\left( c_{j + \frac{n}{2} }\right) + u\left( d_{j + \frac{n}{2} }\right) &= 0 \qquad && (j \in \mathbb{Z}_n) ,\\
        \label{presub_aux2} u\left( c_{j + \frac{n}{2} }\right) + u\left( d_{j + \frac{n}{2} }\right) + u\left( e_{j - \gamma} \right) &= 0 \qquad && (j \in \mathbb{Z}_n) ,\\
        \label{presub_aux3} u\left( d_{j + \gamma }\right) + u\left( e_{j + \frac{n}{2} }\right) + u\left( b_{j + \gamma} \right) &= 0 \qquad && (j \in \mathbb{Z}_n) .
    \end{alignat}
    By combining Equations \eqref{presub_aux1} and \eqref{presub_aux2}, we get $u\left( e_{j - \gamma} \right) = u\left( a_{j - \gamma} \right)$, i.e., $u\left(e_j\right) = u\left(a_j \right)$, for any $j \in \mathbb{Z}_n$. Furthermore, Equation \eqref{presub_aux2} is equivalent to
    \begin{alignat}{2}
        \label{presub_aux4} u\left( c_{j}\right) + u\left( d_{j}\right) + u\left( e_{j + \frac{n}{2} - \gamma} \right) &= 0 \qquad && (j \in \mathbb{Z}_n) ,
    \end{alignat}
    while Equation \eqref{presub_aux3} is equivalent to
    \begin{alignat}{2}
        \label{presub_aux5} u\left( d_{j}\right) + u\left( e_{j + \frac{n}{2} - \gamma}\right) + u\left( b_{j} \right) &= 0 \qquad && (j \in \mathbb{Z}_n) .
    \end{alignat}
    Now, by combining Equations \eqref{presub_aux4} and \eqref{presub_aux5}, we conclude that $u\left(c_j\right) = u\left(b_j\right)$ holds for each $j \in \mathbb{Z}_n$.

    Since $u\left(e_j\right) = u\left(a_j \right)$ and $u\left(c_j\right) = u\left(b_j\right)$, Equation \eqref{presub_aux2} becomes
    \begin{alignat*}{2}
        u\left( b_{j + \frac{n}{2} }\right) + u\left( d_{j + \frac{n}{2} }\right) + u\left( a_{j - \gamma} \right) &= 0 \qquad && (j \in \mathbb{Z}_n) ,
    \end{alignat*}
    hence the system of Equations \eqref{presub_aux1}--\eqref{presub_aux3} is equivalent to the system
    \begin{alignat}{2}
        \label{presub_aux6} u\left(c_j\right) &= u\left(b_j\right) \qquad && (j \in \mathbb{Z}_n) ,\\
        \label{presub_aux7} u\left(e_j\right) &= u\left(a_j \right) \qquad && (j \in \mathbb{Z}_n) ,\\
        \label{presub_aux8} u\left(d_j\right) &= - u\left(a_{j + \frac{n}{2} - \gamma} \right) - u\left(b_j \right) \qquad && (j \in \mathbb{Z}_n) .
    \end{alignat}

    Suppose that $u_0 \in \mathbb{R}^{V(G_1)}$ is any vector that satisfies Equations \eqref{presub_aux6}--\eqref{presub_aux8}. In this case, the local conditions corresponding to all the orbits of $G_1$ different from $c, d, e$ become identical to those corresponding to the orbits of graph $G$. Therefore, if we let $w_0 \in \mathbb{R}^{V(G)}$ be the restriction of $u_0$ to $V(G)$, then it follows that $u_0 \in \mathcal{N}(G_1)$ holds if and only if $w_0 \in \mathcal{N}(G)$. Due to $\dim \mathcal{N}(G) = 1$, we observe that $\dim \mathcal{N}(G_1) = 1$.

    Since $G$ is a nut graph, we know that there exists a full vector $w_0 \in \mathcal{N}(G)$, and it is possible to extend $w_0$ to a uniquely determined vector $u_0 \in \mathbb{R}^{V(G_1)}$ such that Equations \eqref{presub_aux6}--\eqref{presub_aux8} hold for $u = u_0$. It is obvious that $u_0 \in \mathcal{N}(G_1)$ and that $u_0$ has a nonzero entry corresponding to any vertex outside the $c$-, $d$- and $e$-orbits. Due to Equations \eqref{presub_aux6} and \eqref{presub_aux7}, in order to prove that $G_1$ is a nut graph, it is sufficient to show that $u_0\left( d_j \right) \neq 0$ for every $j \in \mathbb{Z}_n$. This follows from Equation~\eqref{presub_aux8} because the $a$- and $b$-orbits of $G$ are of different magnitudes.
\end{proof}

We refer to the construction from Lemma \ref{presub_lemma} as the \emph{pre-subdivision construction}. It can be applied on a $\mathbb{Z}_n$-voltage pregraph that gives rise to a nut graph provided that the pregraph contains two distinct adjacent vertices whose corresponding orbits are of different magnitudes. In fact, this condition holds for all the cubic $7$-circulant nut graphs and cubic $11$-circulant nut graphs from Sections~\ref{sc_7_circ} and \ref{sc_11_circ}, respectively. Besides, if the starting $\mathbb{Z}_n$-voltage pregraph $(X, \varphi)$ satisfies the said condition, then so does the constructed $\mathbb{Z}_n$-voltage pregraph $(X_1, \varphi_1)$. Hence, in this case, the pre-subdivision construction can be iteratively applied on a given starting pregraph. With this in mind, we complete the proof of Theorem \ref{main_th} as follows.

\begin{proof}[Proof of Theorem \ref{main_th}]
We observe that Theorem \ref{circulant_nut_th} and Proposition \ref{nonexist} yield that there exists no cubic $\ell$-circulant nut graph for $\ell \in \{1, 4, 5\}$. Furthermore, the case $\ell = 2$ is resolved in \cite[Proposition 9]{DaBaPiZi2024}. Thus, in order to complete the proof, we only need to show that there exist infinitely many cubic $\ell$-circulant nut graphs for each $\ell \in \{3, 6, 7 \}$ or $\ell \ge 9$.

By Proposition \ref{circ_7_prop}, there are infinitely many cubic $\mathbb{Z}_n$-voltage pregraphs of order seven that give rise to a cubic $7$-circulant nut graph. Each of them satisfies the conditions from Lemma \ref{presub_lemma}, which means that by applying the pre-subdivision construction $t \in \mathbb{N}_0$ times, we obtain infinitely many cubic $(7 + 3t)$-circulant nut graphs. Thus, there are infinitely many cubic $\ell$-circulant nut graphs for each $\ell \in \{7, 10, 13, 16, \ldots \}$. By analogy, it is possible to apply the pre-subdivision construction on the $G^{(11)}(n; \alpha, \beta)$ graphs from Proposition \ref{circ_11_prop}. Therefore, there exist infinitely many cubic $\ell$-circulant nut graphs for each $\ell \in \{11, 14, 17, 20, \ldots \}$. Finally, as shown in \cite[Theorems 2, 11 and 14]{DaBaPiZi2024}, there are infinitely many cubic $\mathbb{Z}_n$-voltage pregraphs of order three that give rise to a nut graph. All of these pregraphs satisfy the conditions from Lemma \ref{presub_lemma}, hence the pre-subdivision construction can now be applied to prove the existence of infinitely many cubic $\ell$-circulant nut graphs for each $\ell \in \{3, 6, 9, 12, \ldots \}$.
\end{proof}

\section{Some open problems}\label{sc_conclusion}
% The final section should serve as a conclusion and it should contain all the left over open problems, most noticeably the polycirculant existence problem for k=8.

We end the paper by disclosing several open problems regarding cubic polycirculant nut graphs. Theorem \ref{main_th} resolves the cubic $\ell$-circulant nut graph existence problem for each $\ell \in \mathbb{N}$ with the exception of $\ell = 8$. The computational results obtained in \cite{GitHub} suggest that these graphs might not exist. This leads us to the following conjecture.

\begin{conjecture}
    There exists no cubic $8$-circulant nut graph.    
\end{conjecture}

Another existence problem tightly connected to Theorem \ref{main_th} is the $d$-regular polycirculant nut graph degree--order--orbit existence problem. In other words, what are all the possible orders that a $d$-regular $\ell$-circulant nut graph can have for a given $d, \ell \in \mathbb{N}$? Let $\mathfrak{N}^d_\ell \subseteq \mathbb{N}$ denote the set of all the $n \in \mathbb{N}$ for which there exists a $d$-regular $\ell$-circulant nut graph of order $n\ell$. The degree--order--orbit existence problem can now be phrased as follows.

\begin{problem}
    Determine $\mathfrak{N}^d_\ell$ for every $d, \ell \in \mathbb{N}$.
\end{problem}

From the results of the present paper, we may conclude that
\[
    \mathfrak{N}^3_1 = \mathfrak{N}^3_2 = \mathfrak{N}^3_4 = \mathfrak{N}^3_5 = \varnothing,
\]
while
\[
    \mathfrak{N}^3_7, \mathfrak{N}^3_{10}, \mathfrak{N}^3_{13}, \ldots \supseteq \{4, 6, 8, 10, 12, \ldots \}
\]
and
\[
    \mathfrak{N}^3_{11}, \mathfrak{N}^3_{14}, \mathfrak{N}^3_{17}, \ldots \supseteq \{6, 10, 14, 18, 22, \ldots \}.
\]
The relation $n \in \mathfrak{N}^3_\ell$ is left to be inspected for certain remaining even $n$ and various $\ell$, and for all the odd $n$ whenever $\ell$ is even.

\section*{Acknowledgements}

The authors would like to express their gratitude to Tomaž Pisanski for suggesting the idea behind the research topic. The work of Nino Bašić is supported in part by the Slovenian Research Agency (research program P1-0294). Ivan Damnjanović is supported by the Science Fund of the Republic of Serbia, grant \#6767, Lazy walk counts and spectral radius of threshold graphs --- LZWK.


\begin{thebibliography}{99}

\bibitem{GitHub}
N.\ Bašić and I.\ Damnjanović, On cubic polycirculant nut graphs:\ supplementary material (GitHub repository), \url{https://github.com/nbasic/cubic-polycirculant-nuts}.

\bibitem{BaFo2024} N.\ Bašić and P.\ W.\ Fowler, Nut graphs with a given automorphism group, 2024, \arxiv{2405.04117}{math.CO}.

\bibitem{BaFoPi2024} N.\ Bašić, P.\ W.\ Fowler and T.\ Pisanski, Vertex and edge orbits in nut graphs, \emph{Electron.\ J.\ Comb.\/}\ {\bf 31(2)} (2024), \#P2.38, \doi{10.37236/12619}.

\bibitem{BaKnSk2022}
N.\ Bašić, M.\ Knor and R.\ Škrekovski, On $12$-regular nut graphs, {\em Art Discrete Appl.\ Math.\/} {\bf 5(2)} (2022), \#P2.01, \doi{10.26493/2590-9770.1403.1b1}.

\bibitem{BrCleePis2013}
G.\ Brinkmann, N.\ Van Cleemput and T.\ Pisanski, Generation of various classes of trivalent graphs, {\em Theor.\ Comput.\ Sci.\/} {\bf 502} (2013), 16--29, \doi{10.1016/j.tcs.2012.01.018}.

\bibitem{CoFoGo2018} K.\ Coolsaet, P.\ W.\ Fowler and J.\ Goedgebeur, Generation and properties of nut graphs, {\em MATCH Commun.\ Math.\ Comput.\ Chem.\/} {\bf 80} (2018), 423--444, \url{https://match.pmf.kg.ac.rs/electronic_versions/Match80/n2/match80n2_423-444.pdf}.

\bibitem{Cvetkovic1995}
D.\ Cvetković, M.\ Doob and H.\ Sachs, \emph{Spectra of graphs:\ Theory and applications}, Leipzig:\ J.\ A.\ Barth Verlag, 1995.

\bibitem{Damnjanovic2023_ARX} I.\ Damnjanović, A note on Cayley nut graphs whose degree is divisible by four, 2023, \arxiv{2305.18658}{math.CO}.

\bibitem{Damnjanovic2023_FIL} I.\ Damnjanović, Two families of circulant nut graphs, {\em Filomat} {\bf 37(24)} (2023), 8331--8360, \doi{10.2298/FIL2324331D}.

\bibitem{Damnjanovic2024_AMC} I.\ Damnjanović, Complete resolution of the circulant nut graph order--degree existence problem, {\em Ars Math.\ Contemp.\/} {\bf 24(4)} (2024), \#P4.03, \doi{10.26493/1855-3974.3009.6df}.

\bibitem{Damnjanovic2024_DMC} I.\ Damnjanović, On the null spaces of quartic circulant graphs, {\em Discrete Math.\ Chem.\/} (2024), in press.

\bibitem{DaBaPiZi2024} I.\ Damnjanović, N.\ Bašić, T.\ Pisanski and A.\ Žitnik, Classification of cubic tricirculant nut graphs, {\em Electron.\ J.\ Comb.\/} {\bf 31(2)} (2024), \#2.31, \doi{10.37236/12668}.

\bibitem{DaSt2022}
I.\ Damnjanović and D.\ Stevanović, On circulant nut graphs, {\em Linear Algebra Appl.\/} {\bf 633} (2022), 127--151, \doi{10.1016/j.laa.2021.10.006}.

\bibitem{FoGaGoPiSc2020} P.\ W.\ Fowler, J.\ B.\ Gauci, J.\ Goedgebeur, T.\ Pisanski and I.\ Sciriha, Existence of regular nut graphs for degree at most $11$, {\em Discuss.\ Math.\ Graph Theory} {\bf 40} (2020), 533--557, \doi{10.7151/dmgt.2283}.

\bibitem{FoPiToBoSc2014} P.\ W.\ Fowler, B.\ T.\ Pickup, T.\ Z.\ Todorova, M.\ Borg and I.\ Sciriha, Omni-conducting and omni-insulating molecules, {\em J.\ Chem.\ Phys.\/} {\bf 140(5)} (2014), 054115, \doi{10.1063/1.4863559}.

\bibitem{FoPiBa2021} P.\ W.\ Fowler, T.\ Pisanski and N.\ Bašić, Charting the space of chemical nut graphs, {\em MATCH Commun.\ Math.\ Comput.\ Chem.\/} {\bf 86(3)} (2021), 519--538, \url{https://match.pmf.kg.ac.rs/electronic_versions/Match86/n3/match86n3_519-538.pdf}.

\bibitem{GaPiSc2023} J.\ B.\ Gauci, T.\ Pisanski and I.\ Sciriha, Existence of regular nut graphs and the Fowler construction, {\em Appl.\ Anal.\ Discrete Math.\/} {\bf 17(2)} (2023), 321--333, \doi{10.2298/AADM190517028G}.

\bibitem{Gray} R.\ M.\ Gray, Toeplitz and circulant matrices:\ a review, {\em Found.\ Trends Commun.\ Inf.\/} Theory {\bf 2} (2006), 155--239, \url{https://ee.stanford.edu/~gray/CIT006-journal.pdf}.

\bibitem{Jameson} G.\ J.\ O.\ Jameson, The cyclotomic polynomials, \url{https://www.maths.lancs.ac.uk/~jameson/cyp.pdf}.

\bibitem{MaMaPo2004} A.\ Malnič, D.\ Marušič and P.\ Potočnik, Elementary Abelian covers of graphs, {\em J.\ Algebraic Combin.\/} {\bf 20} (2004), 71--97, \doi{10.1023/B:JACO.0000047294.42633.25}.

\bibitem{MaNeSko2000} A.\ Malnič, R.\ Nedela and M.\ Škoviera, Lifting graph automorphisms by voltage assignments, {\em European J.\ Combin.\/} {\bf 21} (2000), 927--947, \doi{10.1006/eujc.2000.0390}.

\bibitem{McKayPip2014}
B.\ D.\ McKay and A.\ Piperno, Practical graph isomorphism, II, {\em J.\ Symb.\ Comput.\/}\ {\bf 60} (2014), 94--112, \doi{10.1016/j.jsc.2013.09.003}.

\bibitem{Pisanski2007} T.\ Pisanski, A classification of cubic bicirculants, {\em Discrete Math.\/} {\bf 307(3--5)} (2007), 567--578, \doi{10.1016/j.disc.2005.09.053}.

\bibitem{PiSe2013} T.\ Pisanski and B.\ Servatius, \emph{Configurations from a graphical viewpoint}, Springer, New York, \emph{Birkhäuser Advanced Texts Basler Lehrbücher}, 2013, \doi{10.1007/978-0-8176-8364-1}.

\bibitem{PoTo2020} P.\ Potočnik and M.\ Toledo, Classification of cubic vertex-transitive tricirculants, {\em Ars.\ Math.\ Contemp.\/} {\bf 18} (2020), 1--31, \doi{10.26493/1855-3974.1815.b52}.

\bibitem{Sciriha1997} I.\ Sciriha, On the coefficient of $\lambda$ in the characteristic polynomial of singular graphs, {\em Util.\ Math.\/} {\bf 52} (1997), 97--111.

\bibitem{Sciriha1998_A} I.\ Sciriha, On singular line graphs of trees, {\em Congr.\ Numerantium} {\bf 135} (1998), 73--91.

\bibitem{Sciriha1998_B} I.\ Sciriha, On the construction of graphs of nullity one, {\em Discrete Math.\/} {\bf 181(1--3)} (1998), 193--211, \doi{10.1016/S0012-365X(97)00036-8}.

\bibitem{Sciriha1999} I.\ Sciriha, The two classes of singular line graphs of trees, {\em Rend.\ Semin.\ Mat.\ Messina, Ser.\ II} {\bf 20(5)} (1999), 167--180.

\bibitem{Sciriha2007} I.\ Sciriha, A characterization of singular graphs, {\em Electron.\ J.\ Linear Algebra}, {\bf 16} (2007), 451--462, \url{https://eudml.org/doc/129125}.

\bibitem{Sciriha2008} I.\ Sciriha, Coalesced and embedded nut graphs in singular graphs, {\em Ars Math.\ Contemp.\/} {\bf 1} (2008), 20--31, \doi{10.26493/1855-3974.20.7cc}.

\bibitem{ScFa2021} I.\ Sciriha and A.\ Farrugia, \emph{From nut graphs to molecular structure and conductivity}, University of Kragujevac, Kragujevac, volume 23 of \emph{Mathematical chemistry monographs}, 2021.

\bibitem{ScFo2007} I.\ Sciriha and P.\ W.\ Fowler, Nonbonding orbitals in fullerenes:\ nuts and cores in singular polyhedral graphs, {\em J.\ Chem.\ Inf.\ Model.\/} {\bf  47(5)} (2007), 1763--1775, \doi{10.1021/ci700097j}.

\bibitem{ScFo2008} I.\ Sciriha and P.\ W.\ Fowler, On nut and core singular fullerenes, {\em Discrete Math.\/} {\bf 308(2--3)} (2008), 267--276, \doi{10.1016/j.disc.2006.11.040}.

\bibitem{ScGu1998} I.\ Sciriha and I.\ Gutman, Nut graphs: maximally extending cores, {\em Util.\ Math.\/} {\bf 54} (1998), 257--272.

\bibitem{VanCleemput2014_A} N.\ Van Cleemput, Sequence A243391, in:\ {\em The On-Line Encyclopedia of Integer Sequences}, 2014, \url{https://oeis.org/A243391}.

\bibitem{VanCleemput2014_B} N.\ Van Cleemput, Sequence A243393, in:\ {\em The On-Line Encyclopedia of Integer Sequences}, 2014, \url{https://oeis.org/A243393}.

\bibitem{SageMath} The Sage Developers, SageMath, the Sage Mathematics Software System (Version 9.5), 2022, \url{https://www.sagemath.org}.



\end{thebibliography}
\end{document}